\documentclass[a4paper,10pt]{article}
\usepackage[english,frenchb]{babel}
\usepackage[T1]{fontenc}
\usepackage[utf8]{inputenc}
\usepackage{amsmath}
\usepackage{array}
\usepackage{amsthm}
\usepackage{amssymb}
\input xy
\xyoption{all}

\newcommand{\A}{{\mathcal{A}}}

\newcommand{\C}{{\mathcal{C}}}
\newcommand{\D}{{\mathcal{D}}}

\newcommand{\G}{{\mathcal{G}}}

\newcommand{\col}{{\rm colim}\,}

\newcommand{\kk}{{\Bbbk}}
\newcommand{\md}{\text{-}\mathbf{Mod}}
\newcommand{\mdf}{\text{-}\mathbf{mod}}
\newcommand{\pol}{\mathbf{Pol}}
\newcommand{\gr}{\mathbf{gr}}
\newcommand{\abl}{\mathfrak{a}}
\newcommand{\fct}{\mathbf{Fct}}

\title{Sur l'homologie des groupes d'automorphismes des groupes libres \`a coefficients polynomiaux}
\author{Aur\'elien Djament\thanks{CNRS, laboratoire de math\'ematiques Jean Leray, Nantes ; aurelien.djament@univ-nantes.fr}\; et Christine Vespa\thanks{Institut de Recherche Math\'ematique Avanc\'ee, universit\'e de Strasbourg, vespa@math.unistra.fr. Cet auteur est partiellement soutenu par le projet ANR-11-BS01-0002 HOGT : \textit{Homotopie, Opérades et Groupes de Grothendieck-Teichm\"uller}.}}

\newtheorem{thi}{Th\'eor\`eme}
\newtheorem{pri}{Proposition}

\newtheorem{thm}{Th\'eor\`eme}[section]
\newtheorem{pr}[thm]{Proposition}
\newtheorem{cor}[thm]{Corollaire}
\newtheorem{lm}[thm]{Lemme}

\theoremstyle{definition}
\newtheorem{defi}[thm]{D\'efinition}

\theoremstyle{remark}
\newtheorem{rem}[thm]{Remarque}

\begin{document}

\maketitle

\begin{abstract}
Nous montrons que l'homologie stable des groupes d'automorphismes des groupes libres à coefficients tordus par un foncteur covariant polynomial est trivial. Pour le foncteur d'abélianiation, qui est polynomial de degré $1$, nous retrouvons par des méthodes algébriques un résultat précédemment obtenu par Hatcher-Wahl, par des méthodes topologiques et géométriques. Pour les coefficients donnés par un foncteur polynomial contravariant se factorisant par l'abélianisation, nous calculons la valeur stable du premier groupe d'homologie des groupes d'automorphismes des groupes libres, qui est généralement non nul.
\smallskip

\begin{center}
\textbf{Abstract}
\end{center}
We prove that the stable homology of automorphism groups of free groups with twisted coefficients given by a polynomial covariant functor is trivial. For the abelianization functor, which is polynomial of degree $1$, we recover by algebraic methods a result previously obtained by Hatcher-Wahl by topological and geometrical methods. For coefficients given by a contravariant polynomial functor factorizing through the abelianization, we compute the stable value of the first homology group of automorphism groups of free groups, which is generally  nonzero.

\end{abstract}

\smallskip

\section*{Introduction}

Tandis que l'étude de la structure des groupes d'automorphismes ${\rm Aut}\,(\mathbb{Z}^{*n})$ des groupes libres commença dès les années 1920 (cf. par exemple l'article \cite{Ni}, dans lequel Nielsen en donne une présentation par générateurs et relations), la compréhension de leur homologie s'avéra nettement plus ardue : aucun résultat significatif ne semble avoir été obtenu avant les années 1980. Dans les années 1990, la stabilité homologique a été démontrée pour ces groupes : dans  \cite{HVo} Hatcher et Vogtmann établissent que pour tout entier $i$, le morphisme canonique $H_i({\rm Aut}\,(\mathbb{Z}^{*n}); \mathbb{Z})\to H_i({\rm Aut}\,(\mathbb{Z}^{*n+1}); \mathbb{Z})$ est un isomorphisme pour $n \geq 2i+3$. Quant au calcul de $H_i({\rm Aut}\,(\mathbb{Z}^{*n}); \mathbb{Z})$, hormis pour $i$ ou $n$ très petit, il demeura très mystérieux jusqu'à ce que Galatius démontre dans \cite{Gal} le résultat remarquable suivant : l'inclusion évidente du groupe symétrique $\mathfrak{S}_n$ dans ${\rm Aut}\,(\mathbb{Z}^{*n})$ induit stablement un isomorphisme en homologie à coefficients entiers --- {\em stablement} signifiant : lorsqu'on passe à la colimite sur $n$ (ou pour $n\geq 2i+3$, $i$ désignant le degré homologique). Rappelons que le calcul de l'homologie des groupes symétriques est bien plus ancien ; il est dû à Nakaoka (cf. \cite{Nak}).

  Le théorème de Galatius ne traite que de l'homologie à coefficients {\em constants} des groupes ${\rm Aut}\,(\mathbb{Z}^{*n})$. Il est également très naturel de s'intéresser à l'homologie (ou à la cohomologie) de ces groupes à coefficients dans des représentations remarquables, comme l'abélianisation du groupe libre $\mathbb{Z}^{*n}$ ou des représentations obtenues en appliquant à celle-ci un foncteur ou un bifoncteur sur les groupes abéliens libres.

Le résultat principal du présent article est le suivant :
\begin{thi}\label{thfi}
 Soit $F$ un foncteur polynomial réduit (i.e. nul sur le groupe trivial) de la catégorie $\mathbf{gr}$ des groupes libres de rang fini vers la catégorie $\mathbf{Ab}$ des groupes abéliens. Alors
$$\underset{n\in\mathbb{N}}{\col}H_*\big({\rm Aut}\,(\mathbb{Z}^{*n});F(\mathbb{Z}^{*n})\big)=0.$$ 
\end{thi}
\noindent (La notion de foncteur polynomial est introduite et étudiée dans ce contexte de groupes libres en début d'article ; des exemples typiques sont les puissances tensorielles du foncteur d'abélianisation.)

Lorsque $F$ est le foncteur d'abélianisation, ce résultat avait déjà été démontré, par des méthodes totalement différentes, dans des travaux d'Hatcher et Wahl, où ils obtiennent également un résultat de stabilité (voir \cite{HW} et son erratum \cite{HW-erra}, ainsi que \cite{HW10}). En fait, on peut obtenir le théorème~\ref{thfi} pour toutes les puissances tensorielles de l'abélianiation comme conséquence de l'article \cite{HV04} d'Hatcher et Vogtmann (voir aussi son erratum \cite{HVW} avec Wahl), comme l'a observé Randal-Williams dans \cite{RW}. Tous ces travaux reposent sur des considérations de topologie différentielle. Par une approche encore indépendante (purement algébrique), Satoh a également étudié de tels groupes d'homologie, rendant certains d'entre eux accessibles au calcul (y compris dans le cas instable) en degré homologique $1$ ou $2$ --- voir \cite{Sat1} et \cite{Sat2}.

Alors que le théorème de Galatius affirme que, stablement, l'homologie à coefficients entiers des groupes 
${\rm Aut}\,(\mathbb{Z}^{*n})$ est la même que celle des groupes symétriques, le théorème~\ref{thfi} nous apprend que ceci n'est plus le cas pour l'homologie à coefficients tordus. En effet, dans \cite{Bet-sym} Betley montre que l'homologie stable des groupes symétriques à coefficients dans un foncteur polynomial réduit est loin d'être nulle. Néanmoins, le théorème~\ref{thfi} peut être conçu comme un relèvement aux groupes d'automorphismes des groupes libres du théorème dû à Betley (cf. \cite{Bet2}) selon lequel l'homologie des groupes linéaires sur $\mathbb{Z}$ (cela vaut d'ailleurs pour tout anneau) à coefficients tordus par un foncteur polynomial réduit est stablement nulle. Dans cette situation, considérer un foncteur covariant ou contravariant ne change rien, puisque l'involution des groupes linéaires donnée par la transposée de l'inverse permet de passer d'une situation à l'autre (en revanche, considérer des {\em bi}foncteurs polynomiaux réduits donne lieu à une homologie généralement non nulle ; pour cette généralisation remarquable des résultats de Betley, voir \cite{Bet} ou l'appendice de \cite{FFSS} sur un corps fini et \cite{Sco} pour le cas général). Pour les groupes d'automorphismes des groupes libres, il n'en est pas de même, l'involution en question ne s'y relevant pas ; de fait, en s'appuyant sur le théorème de Betley susmentionné, on parvient à montrer :

\begin{pri}\label{thi2}
 Soit $F$ un foncteur polynomial réduit contravariant de la catégorie $\mathbf{ab}$ des groupes abéliens libres de rang fini vers $\mathbf{Ab}$. Il existe un isomorphisme naturel
$$\underset{n\in\mathbb{N}}{\col} H_1({\rm Aut}\,(\mathbb{Z}^{*n});F(\mathbb{Z}^n))\simeq F\underset{\mathbf{ab}}{\otimes}\mathrm{Id}$$
où ${\rm Aut}\,(\mathbb{Z}^{*n})$ opère sur $F(\mathbb{Z}^n)$ via la projection sur $GL_n(\mathbb{Z})$ et $\mathrm{Id} : \mathbf{ab}\to\mathbf{Ab}$ désigne le foncteur d'inclusion.
\end{pri}
 
Ce résultat est déjà essentiellement présent dans le travail \cite{K-Magnus} de Kawazumi ; il montre que, dès le degré homologique $1$, la situation diffère profondément entre foncteurs polynomiaux contravariants et covariants pour l'homologie stable des groupes d'automorphismes des groupes libres. Pour l'instant, nous ne savons pas traiter le cas du grand degré homologique pour les foncteurs contravariants ; à plus forte raison, la situation pour les bifoncteurs demeure très mystérieuse. La pertinence d'une telle généralisation aux bifoncteurs est illustrée par \cite{K-Magnus}, où Kawazumi introduit des classes de cohomologie reliées à la structure fine des groupes d'automorphismes des groupes libres appartenant à $H^i({\rm Aut}\,(\mathbb{Z}^{*n});{\rm Hom}\,(A,A^{\otimes i+1}))$, où $A$ désigne l'abélianisation du groupe $\mathbb{Z}^{*n}$.

\medskip

Venons-en aux méthodes que nous utilisons pour démontrer nos résultats.

 La stratégie de la preuve du théorème~\ref{thfi} est la suivante. On dispose d'une part de la catégorie  usuelle $\mathbf{gr}$ --- les foncteurs définis sur cette catégorie se prêtent à des calculs d'algèbre homologique --- et d'autre part d'une catégorie de groupes libres de rang fini au\-xi\-li\-aire $\G$ dont l'homologie calcule $\underset{n\in\mathbb{N}}{\col}H_*\big({\rm Aut}\,(\mathbb{Z}^{*n});F(\mathbb{Z}^{*n})\big)$. La catégorie $\G$ a les mêmes objets que $\mathbf{gr}$ mais ses morphismes sont les injections de facteurs libres avec un choix de complément. Pour obtenir le théorème~\ref{thfi}, on compare l'homologie des catégories $\mathbf{gr}$ et $\G$.

Plus précisément, la démonstration du théorème~\ref{thfi} se décompose en trois étapes :
\begin{enumerate}
\item On introduit à la section~\ref{section3} la catégorie $\G$ et le foncteur d'oubli $i : \G\to\mathbf{gr}$. Cette catégorie entre dans le cadre formel introduit dans l'article \cite{DV}, dont un des résultats généraux nous permet d'obtenir (dans la proposition~\ref{pr-dvgl}) un isomorphisme naturel
 $$\underset{n\in\mathbb{N}}{\col} H_*({\rm Aut}\,(\mathbb{Z}^{*n});F(\mathbb{Z}^{*n}))\xrightarrow{\simeq} H_*(\G\times G_\infty;F)$$
pour tout foncteur $F$ défini sur $\G$, où le groupe $G_\infty=\underset{n\in\mathbb{N}}{\col} {\rm Aut}\,(\mathbb{Z}^{*n})$ opère trivialement sur $F$. 

Cependant, les groupes d'homologie $ H_*(\G\times G_\infty;F)$ ne sont pas accessibles par un calcul direct.
\item L'étude homologique de la catégorie $\G$ repose sur des investigations préliminaires dans la catégorie $\mathbf{gr}$, qu'on peut scinder en deux principales étapes :
\begin{itemize}
\item  dans la section~\ref{section4}, on observe que le foncteur d'abélianisation $\mathfrak{a} : \mathbf{gr}\to\mathbf{Ab}$ possède une résolution projective explicite, que procure la résolution barre sur un groupe libre. On en tire, par la considération d'une homotopie explicite, un critère d'annulation de certains groupes d'homologie du type ${\rm Tor}^\mathbf{gr}_*(F,\mathfrak{a})$ (proposition~\ref{pr-ancoabst}).
\item Pour déduire de cette propriété du foncteur d'abélianisation des propriétés générales sur les foncteurs polynomiaux depuis la catégorie $\mathbf{gr}$, nous étudions, dans la section~\ref{section2}, la structure de ces foncteurs polynomiaux. En particulier, nous montrons qu'un tel foncteur  s'obtient par extensions successives de foncteurs (polynomiaux) se factorisant par le foncteur d'abélianisation. Ce résultat, dont la démonstration est indépendante des considérations d'homologie des foncteurs précitées, possède également un intérêt intrinsèque.
\end{itemize}
\item Dans la section~\ref{section5} nous comparons l'homologie de $\G$ et celle de $\mathbf{gr}$. Plus précisément, on montre que le foncteur  $i : \G\to\gr$ induit un isomorphisme 
$$H_*(\G;i^*F)\xrightarrow{\simeq} H_*(\gr;F)$$
pour $F$ un foncteur polynomial sur $\gr$ ; l'homologie $H_*(\gr;F)$ est nulle si $F$ est réduit car la catégorie $\gr$ possède un objet nul.

On établit cet isomorphisme à partir de la suite spectrale de Grothendieck dérivant l'extension de Kan à gauche du foncteur $i$ et des considérations d'algèbre homologique sur $\gr$ susmentionnées.

\end{enumerate}

Quant à la proposition~\ref{thi2}, nous la démontrons d'une manière différente, en u\-ti\-li\-sant la structure de l'abélianisation du noyau $IA_n$
de l'épimorphisme canonique ${\rm Aut}\,(\mathbb{Z}^{*n})\twoheadrightarrow GL_n(\mathbb{Z})$, donnée par exemple dans \cite{K-Magnus}, et les résultats précités sur l'homologie stable des groupes $GL_n(\mathbb{Z})$ à coefficients dans un bifoncteur polynomial. Comme l'homologie des $IA_n$ n'est pas connue au-delà du degré $1$ (des résultats partiels en degré $2$ ont toutefois été obtenus par Pettet dans \cite{Pet}), nous ne pouvons pas encore aller plus loin par cette méthode. Cependant nous pensons que l'approche fonctorielle devrait permettre d'aborder ces questions. Elle pourrait notamment clarifier et aider à démontrer les calculs conjecturaux que donne Randal-Williams dans \cite{RW}.

\paragraph*{Remerciements} Nous remercions Benoît Fresse dont les notes de groupe de travail {\em Applications polynomiales et foncteurs polynomiaux} ont été une source d'inspiration à la généralisation présentée dans la section~\ref{section2}.
Ces notes sont issues d'un cours de Teimuraz Pirashvili, que nous remercions également pour de nombreuses discussions fructueuses sur les foncteurs polynomiaux et l'homologie des foncteurs et pour nous avoir indiqué une difficulté qui nous avait échappé dans le diagramme de recollement du théorème \ref{th-pirarec} et qui donne lieu à la remarque~\ref{Rem-Pira}.

Nous sommes aussi reconnaissants envers Manfred Hartl, Gaël Collinet et Vincent Franjou pour des conversations dont ce travail a bénéficié.

Une partie de ces travaux ayant été effectuée au Max-Planck-Institut f\"ur Mathematik de Bonn, les  auteurs remercient cet institut pour son hospitalité.

\section{Notations et rappels}
Si $\C$ est une catégorie pointée (i.e. possédant un objet nul) ayant des coproduits finis, pour $E\in {\rm Ob}\,\C$, on note $\langle E\rangle_\C$ la sous-catégorie pleine de $\C$ ayant pour objets les sommes finies de copies de $E$.

Pour tout anneau $A$, on désigne par $A\md$ la catégorie des $A$-modules à gauche et par $A\mdf$ la sous-catégorie pleine des modules projectifs de type fini.

Le symbole $\kk$ désigne soit $\mathbb{Z}$, soit un corps premier.

On note $\mathbf{Gr}$ (respectivement $\mathbf{Ab}$) la catégorie des groupes (resp. des groupes abéliens) et $\gr$ (resp. $\mathbf{ab}$) la sous-catégorie pleine des groupes libres de type fini (resp. des groupes abéliens libres de rang fini). Autrement dit, $\gr=\langle\mathbb{Z}\rangle_{\mathbf{Gr}}$ et $\mathbf{ab}=\langle\mathbb{Z}\rangle_{\mathbf{Ab}}=\mathbb{Z}\mdf$.

On désigne par $\abl : \gr\to\kk\md$ le foncteur d'abélianisation tensorisée par~$\kk$.

\subsection{Catégories de foncteurs}
Si $\C$ est une catégorie (essentiellement) petite et $\A$ une catégorie, on note $\fct(\C,\A)$ la catégorie des foncteurs de $\C$ vers $\A$.

La catégorie $\fct(\C,\kk\md)$ est une catégorie abélienne avec limites et colimites se calculant au but ; elle possède assez d'objets injectifs et projectifs. Précisément, pour tout objet $c$ de $\C$, le foncteur $P^\C_c:=\kk[\C(c,-)]$ (on omettra souvent l'exposant $\C$ s'il n'y a pas d'ambiguïté possible) représente l'évaluation en $c$ grâce au lemme de Yoneda, il est donc projectif, et l'ensemble de ces foncteurs lorsque $c$ parcourt ${\rm Ob}\,\C$ engendre la catégorie $\fct(\C,\kk\md)$. En particulier, tout foncteur de $\fct(\C,\kk\md)$ possède une résolution par des sommes directes de projectifs $P^\C_c$.
On peut donc faire de l'algèbre homologique dans cette catégorie comme dans les catégories de modules ; on dispose notamment d'une notion d'{\em homologie de $\C$ à coefficients dans un foncteur} $F\in {\rm Ob}\,\fct(\C,\kk\md)$ --- cette homologie $H_*(\C;F)$ n'est autre que l'évaluation en $F$ des foncteurs dérivés à gauche du foncteur $\underset{\C}{\col} : \fct(\C,\kk\md)\to\kk\md$. On dispose également d'une notion de produit tensoriel au-dessus de $\C$, qui fournit un foncteur $-\underset{\C}{\otimes} - : \fct(\C^{op},\kk\md)\times\fct(\C,\kk\md)\to\kk\md$ qui en chaque variable commute aux colimites et peut se dériver (à gauche) pour donner des foncteurs ${\rm Tor}^\C_n : \fct(\C^{op},\kk\md)\times\fct(\C,\kk\md)\to\kk\md$. Le $\kk$-module $H_n(\C;F)$ s'identifie naturellement à ${\rm Tor}^\C_n(\kk,F)$ (où l'on voit $\kk$ comme foncteur constant à gauche). L'{\em homologie de la catégorie} $\C$ (sans plus de précision) est le $\kk$-module gradué $H_*(\C):=H_*(\C;\kk)\simeq {\rm Tor}^\C_*(\kk, \kk)$. On pourra consulter par exemple l'appendice~A.1 de \cite{Dja-JKT} pour davantage de rappels à ce sujet.

Si $\C$ est pointée, tout objet $F$ de $\fct(\C,\kk\md)$ possède une décomposition canonique $F\simeq F(0)\oplus \bar{F}$ où $F(0)$ désigne le foncteur constant en l'évaluation de $F$ sur l'objet nul et $\bar{F}(c)=Ker\,(F(c)\to F(0))\big(\simeq Coker\,(F(0)\to F(c))\big)$. Le foncteur $\bar{F}$ est {\em réduit}, c'est-à-dire nul en $0$ ; on l'appelle foncteur réduit associé à $F$.

\subsection{Foncteurs polynomiaux}

Soit $\C$ une petite catégorie pointée avec coproduits finis. On dispose d'une notion classique de foncteur polynomial dans $\fct(\C,\kk\md)$. Celle-ci remonte à Eilenberg et Mac Lane (cf. \cite{EML}, chap.~II), au moins dans le cas où $\C$ est la catégorie des groupes abéliens (la définition générale étant la même). On pourra également se reporter à \cite{HV}, §\,1, pour une exposition générale. Rappelons qu'un foncteur polynomial de degré au plus $n-1$ est un foncteur $F$ dont le $n$-ème effet croisé $cr_n(F)$, qui est un multifoncteur en $n$ variables sur $\C$, est nul. Les foncteurs polynomiaux de degré au plus $n$ de $\fct(\C,\kk\md)$ forment une sous-catégorie épaisse stable par limites et colimites ; on la notera $\pol_n(\C)$.

\medskip

Fixons un objet $E$ de $\C$ : les foncteurs polynomiaux sur la catégorie $\langle E\rangle_\C$ constituent un cas particulièrement important (cf. \cite{HV} et \cite{PJ}).
  
Tout foncteur $F : \langle E\rangle_\C\to\kk\md$ possède un plus grand quotient additif (i.e. polynomial de degré au plus $1$) et réduit, qu'on notera $\bar{T}_1(F)$. Il est explicitement donné par $\bar{T}_1(F)(V)=Coker\,(F(p_1)+F(p_2)-F(s) : F(V+V)\to F(V))$ où $+$ désigne le coproduit de $\C$ et $p_1, p_2, s: V+V\to V$ sont respectivement les première projection, deuxième projection et somme. Le $\kk$-module $\Lambda_\C(E):=\bar{T}_1(P_E)(E)$ est un quotient de $\kk[{\rm End}_\C(E)]$ ; on vérifie (cf. \cite{HV}, §\,3.2) que la structure de $\kk$-algèbre sur $\kk[{\rm End}_\C(E)]$ définit par passage au quotient une structure de $\kk$-algèbre sur $\Lambda_\C(E)$. De plus, l'action naturelle à droite de $\kk[{\rm End}_\C(E)]$ sur le foncteur $P_E$ induit une action (à droite) de $\Lambda_\C(E)$ sur le foncteur $\bar{T}_1(P_E)$.

\section{Foncteurs polynomiaux sur $\gr$} \label{section2}

Dans cette section, on donne un résultat de classification général sur les foncteurs polynomiaux (nous restons dans le cadre d'une catégorie source du type $\langle E\rangle_\C$, où $\C$ est une catégorie pointée avec coproduits finis), avant de l'appliquer à la catégorie $\fct(\gr,\kk\md)$. Ces résultats, qui présentent un intérêt intrinsèque, ne seront utilisés que dans la démonstration de la proposition~\ref{praux-fctpol} (uniquement via le corollaire~\ref{cor-recdeg} et la formule~(\ref{eq-alpha})).

On rappelle la définition suivante :

\begin{defi}
Si $R$ est un anneau muni d'une action (à gauche) d'un groupe $G$, on définit l'algèbre tordue de $G$ sur $R$ comme le $R$-module $R[G]$ muni de la multiplication donnée par $r[g].s[h]=(r\,^g\! s).[gh]$ pour tout $(r,s,g,h)\in R\times R\times G\times G$ (où $^g\! s$ désigne l'action de $g$ sur $s$) ; dans le cas de l'action du groupe symétrique $\mathfrak{S}_n$ sur le produit tensoriel $A^{\otimes n}$ de $n$ copies d'un anneau $A$ par permutation des facteurs, on notera $\mathfrak{S}_n\wr A$ l'anneau obtenu, appelé {\em produit en couronne de $A$ par $\mathfrak{S}_n$}. 
\end{defi}

Ainsi, pour tout entier $n$, $\bar{T}_1(P_E)^{\otimes n}$ est un foncteur polynomial de degré (au plus) $n$ réduit muni d'une action de $\mathfrak{S}_n\wr\Lambda_\C(E)$. Pour tout foncteur $F$ sur $\langle E\rangle_\C$, l'effet croisé $cr_n(F)(E,\dots,E)$ est muni d'une action naturelle de $\mathfrak{S}_n\wr\Lambda_\C(E)$, où le groupe symétrique opère par permutation des facteurs et l'action de $\Lambda_\C(E)^{\otimes n}$ est induite par celle de $\kk[{\rm End}_\C(E)]^{\otimes n}$ sur $F(E^{+n })$ (chacun des $n$ facteurs du produit tensoriel opérant par précomposition sur chacun des $n$ facteurs de la somme). On notera ${\rm cr}_n(F)$ le $\mathfrak{S}_n\wr\Lambda_\C(E)$-module $cr_n(F)(E,\dots,E)$.

Le résultat suivant est dû à Pirashvili (\cite{Pira-rec}) dans le cas où $\C$ est une catégorie de modules ; il doit beaucoup à \cite{HV} (qui étudie de façon bien plus précise la structure des foncteurs polynomiaux de degré $2$) pour le passage au cas général.

\begin{thm}\label{th-pirarec} Il existe un diagramme de recollement
$$\xymatrix{\pol_{n-1}(\langle E\rangle_\C)\ar[r]|-{\mathrm{incl}} & \pol_{n}(\langle E\rangle_\C)\ar[r]|-{{\rm cr}_n}\ar@/_/[l]_-l\ar@/^/[l]^-r & \big(\mathfrak{S}_n\wr\Lambda_\C(E)\big)\md\ar@/_/[l]_-{\alpha_n}\ar@/^/[l]^-{\beta_n}
}$$
où $\mathrm{incl}$ est l'inclusion et
$$\alpha_n(M)=\bar{T}_1(P_E)^{\otimes n}\underset{\mathfrak{S}_n\wr\Lambda_\C(E)}{\otimes} M\simeq\big(\bar{T}_1(P_E)^{\otimes n}\underset{\Lambda_\C(E)^{\otimes n}}{\otimes} M\big)_{\mathfrak{S}_n}.$$
\end{thm}

Rappelons ce que signifie que le diagramme précédent de foncteurs entre catégories abéliennes est un diagramme de recollement (cf. par exemple \cite{K2}, §\,2, et \cite{FP} pour une présentation des diagrammes de recollement dans les catégories abéliennes et une utilisation dans les catégories de foncteurs)
\begin{enumerate}
 \item le foncteur incl est adjoint à droite (resp. à gauche) au foncteur $l$ (resp. $r$) ; de plus, l'unité $Id\to r.\mathrm{incl}$ est un isomorphisme ;
\item le foncteur $\mathrm{cr}_n$ est adjoint à droite (resp. à gauche) au foncteur $\alpha_n$ (resp. $\beta_n$) ; de plus, l'unité $Id\to\mathrm{cr}_n .\alpha_n$ est un isomorphisme ;
\item le foncteur incl est pleinement fidèle et son image essentielle est le noyau du foncteur $\mathrm{cr}_n$.
\end{enumerate}
Remarquons que la définition de diagramme de recollement donnée dans \cite{K2} comporte des conditions redondantes. En effet, d'après les propositions $3.4.1$ et $3.4.2$ de \cite{BorceuxI}, la coünité de l'adjonction entre ${\rm cr}_n$ et $\beta_n$ est un isomorphisme si et seulement si l'unité de l'adjonction entre $\mathrm{cr}_n$ et $\alpha_n$ est un isomorphisme. De même pour les adjonctions entre $\mathrm{incl}$, $l$ et $r$.

Rappelons également que ces propriétés impliquent formellement que le foncteur $\mathrm{cr}_n$ induit une équivalence de catégories
$$\pol_{n}(\langle E\rangle_\C)/\pol_{n-1}(\langle E\rangle_\C)\xrightarrow{\simeq}\mathfrak{S}_n\wr\Lambda_\C(E)\md.$$

\begin{proof}
 Pour le cas crucial du degré $1$, voir \cite{HV}, Théorème $3.12$. 
 
 Pour le cas général : comme $\pol_{n-1}(\langle E\rangle_\C)$ est une catégorie de Grothendieck et que $\pol_{n-1}(\langle E\rangle_\C)$ est stable par limites et colimites, on obtient l'existence des adjoints à droite et à gauche au foncteur $\mathrm{incl}$ et on montre facilement que l'unité $Id\to r.\mathrm{incl}$ est un isomorphisme, ce qui démontre le point $1$.
 
Le point $3.$ découle de la définition de $\pol_{n-1}(\langle E\rangle_\C)$ et du fait que c'est une sous-catégorie pleine de $\pol_{n}(\langle E\rangle_\C)$.

Pour la partie droite du diagramme : pour $F \in  \pol_{n}(\langle E\rangle_\C)$, on a un isomorphisme naturel $\mathfrak{S}_n$-équivariant $\mathrm{Hom}(T_1(P_E)^{\otimes n}, F) \simeq F(E^{+ n})$ (qui est induit par le lemme de Yoneda et l'isomorphisme $T_n(\bar{P_E}^{\otimes n}) \simeq T_1(\bar{P_E})^{\otimes n}$ --- voir \cite{HV} pour une démonstration détaillée dans le cas $n=2$). Cet isomorphisme induit un isomorphisme naturel et $\mathfrak{S}_n \wr\Lambda_\C(E)$-équivariant $\mathrm{Hom}(\bar{T_1}(P_E)^{\otimes n}, F) \simeq \mathrm{cr}_n(F)$. On en déduit les isomorphismes naturels :
$$\mathrm{Hom}_{\mathfrak{S}_n \wr\Lambda_\C(E) \md}(M, \mathrm{cr}_n(F)) \simeq \mathrm{Hom}_{\mathfrak{S}_n \wr\Lambda_\C(E)}(M, \mathrm{Hom}_{ \pol_{n}(\langle E\rangle_\C)}(\bar{T_1}(P_E)^{\otimes n}, F))$$
$$ \simeq \mathrm{Hom}_{ \pol_{n}(\langle E\rangle_\C)}(\bar{T}_1(P_E)^{\otimes n}\underset{\mathfrak{S}_n\wr\Lambda_\C(E)}{\otimes} M,F)$$
où la dernière équivalence est l'adjonction usuelle entre le Hom externe et $\otimes$.
Pour $M \in \mathfrak{S}_n \wr\Lambda_\C(E)\md$ on a les isomorphismes naturels :
$$\mathrm{cr}_n\alpha_n(M)=\mathrm{cr}_n\big((\bar{T}_1(P_E)^{\otimes n}\underset{\Lambda_\C(E)^{\otimes n}}{\otimes} M)_{\mathfrak{S}_n} \big) \simeq \bar{T}_1(P_E)(E)^{\otimes n}\underset{\Lambda_\C(E)^{\otimes n}}{\otimes} M \simeq M.$$
On en déduit que l'unité de l'adjonction entre $\mathrm{cr}_n$ et $\alpha_n$ est un isomorphisme.

Enfin, comme les catégories $\pol_{n}(\langle E\rangle_\C)$ et $\big(\mathfrak{S}_n\wr\Lambda_\C(E)\big)\md$ sont des catégories de Grothendieck et que le foncteur $\mathrm{cr}_n$ commute aux colimites ce foncteur possède un adjoint à droite. 

\end{proof}

\begin{rem} \label{Rem-Pira}
L'adjoint à droite $\beta_n$ du foncteur ${\rm cr}_n$ est isomorphe à 
$$M \mapsto \big(\bar{T}_1(P_E)^{\otimes n}\underset{\Lambda_\C(E)^{\otimes n}}{\otimes} M\big)^{\mathfrak{S}_n}$$
{\em lorsque la catégorie source est additive}, mais ce n'est plus le cas en général (notamment pour la catégorie $\mathbf{gr}$ qui nous intéresse).
\end{rem}

\begin{cor}\label{cor-recdeg}
 Soit $F\in {\rm Ob}\,\pol_n(\langle E\rangle_\C)$ ; notons $M$ le $\mathfrak{S}_n\wr\Lambda_\C(E)$-module $\mathrm{cr}_n(F)$. Les noyaux et conoyaux de la coünité $\alpha_n(M)\to F$ et de l'unité $F\to\beta_n(M)$ sont de degré inférieur ou égal à $n-1$. 
\end{cor}

\begin{pr}\label{pr-compdr}
 Soient $F : \C\to\D$ un foncteur entre catégories pointées avec sommes finies qui commute aux sommes finies, $E$ un objet de $\C$ et $E'=F(E)$. On suppose que $F$ est plein et induit un isomorphisme $\C(E,E)\to\D(E',E')$. Alors $F$ induit un isomorphisme $\Lambda_\C(E)\xrightarrow{\simeq}\Lambda_\D(E')$ et une équivalence de catégories :
$$\pol_n(\langle E\rangle_\C)/\pol_{n-1}(\langle E\rangle_\C)\xrightarrow{\simeq}\pol_n(\langle E'\rangle_\D)/\pol_{n-1}(\langle E'\rangle_\D)$$
pour tout $n\in\mathbb{N}$.
\end{pr}

\begin{proof}
 On déduit directement le fait que $F$ induit un isomorphisme d'anneaux $\Lambda_\C(E)\xrightarrow{\simeq}\Lambda_\D(E')$ de la définition, des hypothèses et du lemme des cinq. Le reste en découle par le théorème précédent.
\end{proof}

Le critère de comparaison général qui précède s'applique aux groupes libres, abéliens ou non :

\begin{cor}\label{cor-quotpol}
 Le foncteur d'abélianisation $\gr\to\mathbf{ab}$ induit des équivalences de catégories :
$$\pol_n(\gr)/\pol_{n-1}(\gr)\simeq\pol_n(\mathbf{ab})/\pol_{n-1}(\mathbf{ab})\simeq\kk[\mathfrak{S}_n]\md.$$
\end{cor}

\begin{proof}
 Le foncteur d'abélianisation vérifie les conditions de la proposition précédente. La dernière équivalence est classique et s'obtient à partir du théorème~\ref{th-pirarec} et du calcul $\bar{T}_1(P_{\mathbb{Z}}^{\mathbf{ab}})=\abl$ qui implique $\Lambda_\mathbf{ab}(\mathbb{Z})\simeq\kk$.
\end{proof}

En particulier, on voit que le foncteur $\alpha_n$ est donné, lorsque la catégorie source est $\gr$, par
\begin{equation}\label{eq-alpha}
\alpha_n(M)=\abl^{\otimes n}\underset{\mathfrak{S}_n}{\otimes} M.
\end{equation}

Le corollaire suivant montre que les foncteurs polynomiaux sur $\gr$ ne sont <<~pas loin~>> de se factoriser par l'abélianisation $\gr\to\mathbf{ab}$ : ils peuvent tous s'obtenir par extensions successives de foncteurs possédant une telle factorisation.

\begin{cor}\label{corpolfil}
 Tout foncteur $F\in {\rm Ob}\,\pol_n(\gr)$ possède une filtration
$$0=F_{-1} \subset F_0 \subset F_1\subset\dots\subset F_{n-1}\subset F_n=F$$
telle que chaque sous-quotient $F_i/F_{i-1}$ soit de degré au plus $n-i$ et appartienne à l'image essentielle du foncteur $\fct(\mathbf{ab},\kk\md)\to\fct(\gr,\kk\md)$ de précomposition par l'abélianisation.
\end{cor}

\begin{proof}
On raisonne par récurrence sur le degré $n$ du foncteur $F$. Pour $n=0$ le résultat est trivial. Pour  $n>0$, on note $F_0$ l'image de la co\"unité :
$$\alpha_n(\mathrm{cr}_n(F))\to F.$$
Par le corollaire \ref{cor-recdeg}, le quotient $G=F/F_0$ est dans ${\rm Ob}\,\pol_{n-1}(\gr)$. Par l'hypothèse de récurrence, $G$ possède une filtration 
$$0=G_{-1} \subset G_0 \subset G_1\subset\dots\subset G_{n-1}=G$$
telle que chaque sous-quotient $G_i/G_{i-1}$ soit de degré au plus $n-1-i$ et  appartienne à l'image essentielle du foncteur $\fct(\mathbf{ab},\kk\md)\to\fct(\gr,\kk\md)$ de précomposition par l'abélianisation.

Soit $\pi: F \twoheadrightarrow G=F/F_0$ la projection. Pour $1 \leq i \leq n$, on note $F_i$ l'image réciproque de $G_{i-1}$ par $\pi$. Ceci fournit une filtration de $F$ :
$$0=F_{-1} \subset F_0 \subset F_1\subset\dots\subset F_{n-1}\subset F_n=F.$$
Pour $ 1\leq i\leq n$, le quotient $F_i/F_{i-1} \simeq G_{i-1}/G_{i-2}$ se factorise par l'abélianisation et  
$$\deg(F_i/F_{i-1})=\deg(G_{i-1}/G_{i-2}) \leq (n-1)-(i-1)=n-i$$
où $\deg(F)$ désigne le degré polynomial du foncteur $F$.
De plus, comme $\alpha_n( \mathrm{cr}_n(F))$ se factorise par l'abélianisation d'après (\ref{eq-alpha}), il en est de même de $F_0$ et $\deg(F_0) \leq n$ puisque $F_0$ est un sous-foncteur de $F$.
\end{proof}

\section{La catégorie de groupes libres auxiliaire $\G$} \label{section3}

Dans \cite{DV}, nous avons introduit des axiomes sur une petite catégorie $\C$ per\-met\-tant de relier, d'une part, l'homologie de groupes d'automorphismes d'objets de $\C$, à coefficients tordus par un foncteur $F$ de $\fct(\C,\kk\md)$ (supposant connue l'homologie des mêmes groupes à coefficients constants dans $\kk$), d'autre part l'homologie $H_*(\C;F)$ de la catégorie $\C$ à coefficients dans $F$. La catégorie $\gr$ ne satisfait toutefois pas à ces axiomes. Pour y remédier, nous introduisons une autre catégorie de groupes libres $\G$ et tirons dans cette section les conclusions de \cite{DV} pour celle-ci, avant d'étudier son lien homologique avec la catégorie $\gr$.
 
\begin{defi} On note $\G$ la catégorie dont les objets sont les groupes libres de type fini et dans laquelle un morphisme $A\to B$ est un couple $(u,H)$ formé d'un monomorphisme de groupes $u : A\to B$ et d'un sous-groupe $H$ de $B$ tels que $B$ soit le produit libre de l'image de $u$ et de $H$. La composée $A\xrightarrow{(u,H)}B\xrightarrow{(v,K)} C$ est par définition $(v\circ u,v(H)*K)$.
\end{defi}

Cette catégorie est reliée à la catégorie usuelle $\gr$ par les deux foncteurs fondamentaux suivants.

\begin{defi}\label{df-fctfond}
\begin{enumerate}
 \item On note $i : \G\to\gr$ le foncteur égal à l'identité sur les objets et associant à un morphisme $(u,H) : A\to B$ de $\G$ le morphisme de groupes $u : A\to B$.
\item On note $\iota : \G^{op}\to\gr$ le foncteur égal à l'identité sur les objets et associant à un morphisme $(u,H) : A\to B$ de $\G$ le morphisme de groupes $B=u(A)*H\twoheadrightarrow u(A)\xrightarrow{u^{-1}}A$ composé de la projection canonique et de l'isomorphisme inverse de l'isomorphisme qu'induit $u$ entre $A$ et son image.
\end{enumerate}
\end{defi}

Noter que ces foncteurs ne sont ni pleins ni fidèles, mais qu'ils induisent des isomorphismes entre les groupes d'automorphismes des objets.

\begin{rem}\label{rq-sco1}
La catégorie $\G$ constitue un analogue non abélien de la catégorie $\mathbf{S}(\mathbb{Z})$ des groupes abéliens libres de rang fini avec pour morphismes les monomorphismes scindés, {\em le scindage étant donné}. Les foncteurs $i$ et $\iota$ sont alors similaires aux foncteurs covariant et contravariant tautologiques de $\mathbf{S}(\mathbb{Z})$ vers $\mathbf{ab}$.

La catégorie $\mathbf{S}(\mathbb{Z})$ est utilisée dans \cite{Dja-JKT} pour montrer les résultats de Scorichenko (\cite{Sco}) sur l'homologie stable des groupes linéaires sur $\mathbb{Z}$ à coefficients polynomiaux, en simplifiant légèrement la méthode de cet auteur (qui consiste à considérer la catégorie des groupes abéliens libres de rang fini avec pour morphismes les monomorphismes scindés, le scindage ne faisant pas partie de la structure).

Ici, contrairement à la situation pour les groupes linéaires, les foncteurs $i$ et $\iota$ ne jouent pas du tout le même rôle ; on en verra une illustration spectaculaire à la fin de cet article.
\end{rem}

\medskip

Le produit libre $*$ fait de $\G$, tout comme de $\gr$, une catégorie monoïdale symétrique, et les foncteurs $i$ et $\iota$ sont monoïdaux. Le groupe trivial $0$, unité de cette structure, est également objet initial de $\G$. En particulier, pour tous objets $A$ et $B$ de $\G$, on dispose d'un morphisme canonique $A=A*0\to A*B$, qui est équivariant relativement aux actions tautologiques de ${\rm Aut}\,(A)$ et ${\rm Aut}\,(A*B)$ et au monomorphisme de groupes canonique ${\rm Aut}\,(A)\to{\rm Aut}\,(A*B)$. Par conséquent, si $F$ est un objet de $\fct(\G,\kk\md)$, on dispose de morphismes naturels
$$H_*({\rm Aut}\,(\mathbb{Z}^{*n});F(\mathbb{Z}^{*n}))\to H_*({\rm Aut}\,(\mathbb{Z}^{*m});F(\mathbb{Z}^{*m}))$$
pour $n\leq m$ (plonger $\mathbb{Z}^{*n}$ dans $\mathbb{Z}^{*m}$ par l'inclusion des $n$ premiers facteurs) ; la colimite de ces groupes lorsque $n$ parcourt $\mathbb{N}$ est appelée {\em homologie stable des groupes d'automorphismes des groupes libres à coefficients tordus par $F$}. On dispose par ailleurs de l'homologie $H_*(\G;F)$ de la catégorie $\G$ à coefficients dans $F$ ; les inclusions des sous-catégories pleines réduites à l'objet $\mathbb{Z}^{*n}$ dans $\G$ permettent de définir un morphisme naturel de l'homologie stable définie précédemment vers $H_*(\G;F)$. Pour $F$ constant, cette homologie est triviale car $\G$ a un objet initial. Par conséquent, le morphisme naturel défini précédemment ne saurait être, en général, un isomorphisme. 
Pour contourner cette difficulté, on considère plutôt l'homologie $H_*(\G\times G_\infty;F)$, où $G_\infty:=\underset{n\in\mathbb{N}}{\col} {\rm Aut}\,(\mathbb{Z}^{*n})$, qui constitue en fait un abus de notation pour $H_*(\G\times G_\infty;\Pi^*F)$ où $\Pi : \G\times G_\infty\to\G$ est le foncteur de projection (autrement dit, on fait agir le groupe $G_\infty$ trivialement sur le foncteur $F$) : il existe également un morphisme naturel de l'homologie stable de groupes à coefficients dans $F$ vers $H_*(\G\times G_\infty;\Pi^*F)$, obtenu cette fois-ci en prenant la colimite des morphismes induits par les foncteurs ${\rm Aut}\,(\mathbb{Z}^{*n})\to\G\times G_\infty$ dont la composante vers $\G$ est la même que précédemment (inclusion pleinement fidèle d'image $\mathbb{Z}^{*n}$) et la composante ${\rm Aut}\,(\mathbb{Z}^{*n})\to G_\infty$ est le morphisme canonique.

Des résultats généraux des deux premières sections de \cite{DV} on tire la proposition suivante :

\begin{pr}\label{pr-dvgl}
 Pour tout foncteur $F\in {\rm Ob}\,\fct(\G,\kk\md)$, le morphisme naturel 
$$\underset{n\in\mathbb{N}}{\col} H_*({\rm Aut}\,(\mathbb{Z}^{*n});F(\mathbb{Z}^{*n}))\to H_*(\G\times G_\infty;F)$$
de $\kk$-modules gradués, où le groupe $G_\infty=\underset{n\in\mathbb{N}}{\col} {\rm Aut}\,(\mathbb{Z}^{*n})$ opère trivialement sur $F$, est un isomorphisme. 
\end{pr}

\begin{proof}
 On note les trois propriétés suivantes :
\begin{enumerate}
 \item tout objet de $\G$ est isomorphe au produit libre d'un nombre fini de copies de $\mathbb{Z}$ ;
\item pour tous objets $A$ et $B$ de $\G$, le groupe d'automorphismes ${\rm Aut}\,(B)$ opère transitivement sur $\G(A,B)$. En effet, si $(u,H)$ et $(v,K)$ sont deux morphismes $A\to B$ de $\G$, les isomorphismes de groupes $B\simeq A*H\simeq A* K$ entraînent $H\simeq K$ et procurent donc un automorphisme $\varphi$ de $B$ tel que $\varphi(H)=K$ et $\varphi\circ u=v$, de sorte que $\varphi\circ (u,H)=(v,K)$. 
\item Pour tous objets $A$ et $B$ de $\G$, le morphisme canonique du groupe ${\rm Aut}\,(B)$ vers le stabilisateur du morphisme canonique $A\to A*B$ de $\G$ sous l'action de ${\rm Aut}\,(A*B)$ est un isomorphisme. Cela découle de ce que ce stabilisateur est l'ensemble des automorphismes $\varphi$ du groupe $A*B$ qui coïncident avec l'identité sur $A$ et tels que $\varphi(B)=B$.
\end{enumerate}

Ces conditions impliquent formellement, d'après \cite{DV} (ou la proposition~1.4 de \cite{Dja-JKT}, qui en reprend les résultats), la proposition.
\end{proof}

En utilisant les résultats remarquables de Galatius (cf. \cite{Gal}) sur l'homologie de $G_\infty$ et la formule de Künneth, on en déduit par exemple :

\begin{cor}\label{cor-dvg}
Lorsque $\kk$ est un corps, on dispose d'un isomorphisme naturel
 $$\underset{n\in\mathbb{N}}{\col} H_*({\rm Aut}\,(\mathbb{Z}^{*n});F(\mathbb{Z}^{*n}))\simeq H_*(\G;F)\otimes H_*(\mathfrak{S}_\infty;\kk).$$

Pour $\kk=\mathbb{Q}$, le morphisme naturel
$$\underset{n\in\mathbb{N}}{\col} H_*({\rm Aut}\,(\mathbb{Z}^{*n});F(\mathbb{Z}^{*n}))\to H_*(\G;F)$$
est un isomorphisme.
\end{cor}

\section{Algèbre homologique dans la catégorie $\gr$} \label{section4}

La catégorie $\fct(\gr,\kk\md)$ se prête à des calculs d'algèbre homologique, et ce de façon beaucoup plus aisée que les catégories $\fct(A\mdf,A\md)$, même lorsque $A$ est l'anneau des entiers ou un corps fini, où les premiers calculs non triviaux s'avèrent déjà délicats (cf. \cite{FPZ} et \cite{FLS} respectivement). Cela tient au fait remarquable, clef de voûte du présent travail, que le foncteur d'abélianisation $\abl\in {\rm Ob}\,\fct(\gr,\kk\md)$ possède une résolution projective explicite, donnée par la résolution barre. Ce fait apparaît dans le travail \cite{PJ} de Jibladze et Pirashvili. Avant de donner explicitement cette résolution, introduisons une simplification de notation. Pour tout entier naturel $n$, on désigne par $P_n$ le foncteur projectif $P^\gr_{\mathbb{Z}^{*n}}$. Ainsi, $P_n(G)\simeq\kk[G^n]$ canoniquement. Comme la catégorie $\gr$ possède des sommes finies, on dispose d'isomorphismes $P_i\otimes P_j\simeq P_{i+j}$, d'où $P_n\simeq P^{\otimes n}$ où $P=P_1$.

\begin{pr}[Cf. \cite{PJ}, proposition~5.1]\label{pr-bar}
 Le foncteur $\abl$ possède une résolution projective :
$$\dots P_{n+1}\xrightarrow{d_n} P_n\to\dots\to P_2\xrightarrow{d_1}P_1$$
où la transformation naturelle $d_n : P_{n+1}\to P_n$ est donnée sur le groupe $G$ par l'application linéaire $\kk[G^{n+1}]\to\kk[G^n]$ telle que
$$d_n([g_1,\dots,g_{n+1}])=[g_2,\dots,g_{n+1}]+\sum_{i=1}^n(-1)^i [g_1,\dots,g_{i-1},g_i g_{i+1},g_{i+2},\dots,g_{n+1}]+(-1)^{n+1}[g_1,\dots,g_n]$$
pour tout $(g_1,\dots,g_{n+1})\in G^{n+1}$.
\end{pr}

\begin{proof}
 La suite de foncteurs de l'énoncé, évaluée sur un groupe $G$, n'est autre que la résolution barre sur ce groupe (dont on a tronqué le degré $0$). Elle calcule donc fonctoriellement l'homologie réduite de $G$ à coefficients dans $\kk$. La conclusion résulte donc de ce que l'homologie réduite à coefficients dans $\kk$ d'un groupe libre est naturellement isomorphe à son abélianisation tensorisée par $\kk$ concentrée en degré~$1$. 
\end{proof}

En tensorisant cette résolution par un projectif $P_r$, on en déduit :

\begin{cor}\label{cor-bar}
 Soient $X\in {\rm Ob}\,\fct(\gr^{op},\kk\md)$ et $r\in\mathbb{N}$. Les groupes de torsion ${\rm Tor}^\gr_*(X,\abl\otimes P_r)$ sont isomorphes à l'homologie du complexe :
$$\dots\to X(n+r+1)\xrightarrow{\delta_n} X(n+r)\to\dots\to X(r+2)\xrightarrow{\delta_1}X(r+1)$$
(on s'autorise à noter $X(i)$ pour $X(\mathbb{Z}^{*i})$) où
$$\delta_n=X(a^{n,r})+\sum_{i=1}^n (-1)^i X(b_i^{n,r})+(-1)^{n+1}X(c^{n,r}),$$
les morphismes $a^{n,r},b_i^{n,r},c^{n,r} : \mathbb{Z}^{*r+n}\to\mathbb{Z}^{*r+n+1}$ étant donnés, via l'identification $\gr(\mathbb{Z}^{*r+n},\mathbb{Z}^{*r+n+1})\simeq (\mathbb{Z}^{*r+n+1})^{n+r}$, par :
$$a^{n,r}=(e_2,\dots,e_{n+r+1})$$
$$b_i^{n,r}=(e_1,\dots,e_{i-1},e_i e_{i+1},e_{i+2},\dots,e_{n+r+1})$$
$$c^{n,r}=(e_1,\dots,e_n,e_{n+2},\dots,e_{n+r+1})$$
où $(e_1,\dots,e_{n+r+1})$ désigne la <<~base~>> tautologique de $\mathbb{Z}^{*r+n+1}$.
\end{cor}

\begin{rem}
 Au lieu de la construction barre, on peut utiliser la construction barre réduite (ou normalisée). Cela fournit une résolution projective de $\abl$ de la forme suivante :
$$\dots\bar{P}^{\otimes n+1}\to\bar{P}^{\otimes n}\to\dots\to\bar{P}.$$
On en déduit en particulier ${\rm Tor}^\gr_i(X,\abl)=0$ lorsque $X$ est un foncteur polynomial de degré inférieur ou égal à $i$. Une propriété analogue vaut pour les groupes d'extensions ; notamment, on voit que ${\rm Ext}^*_\gr(\abl,\abl)$ est réduit à $\kk$ concentré en degré $0$. Ces groupes d'extensions sont des analogues pour la catégorie $\gr$ des auto-extensions du foncteur d'inclusion de $\fct(A\mdf,A\md)$, qui s'identifie à la cohomologie de Mac Lane de $A$. Celle-ci est difficile à calculer (cf. les références \cite{FLS} et \cite{FPZ} précitées), même pour les corps, hormis dans le cas de la caractéristique nulle, où l'on peut également s'appuyer sur la résolution barre. 

Néanmoins, cette variante réduite n'est pas très adaptée à nos considérations ultérieures sur la catégorie $\gr$.
\end{rem}

Le corollaire~\ref{cor-bar} constitue la base du critère d'annulation homologique abstrait suivant.

\begin{pr}\label{pr-ancoabst}
 Soit $X\in {\rm Ob}\,\fct(\gr^{op},\kk\md)$ un foncteur tel qu'existe, pour tous objets $A$ et $T$ de $\gr$, une application linéaire $\xi(A,T) : X(A)\to X(T*A)$ vérifiant les propriétés suivantes :
\begin{enumerate}
 \item pour tous $\varphi : A\to B$ et $T$ dans $\gr$, les composées
$$X(B)\xrightarrow{X(\varphi)} X(A)\xrightarrow{\xi(A,T)}X(T*A)$$
et
$$X(B)\xrightarrow{\xi(B,T)} X(T*B)\xrightarrow{X(T*\varphi)}X(T*A)$$
coïncident ;
\item la composée
$$X(A)\xrightarrow{\xi(A,T)}X(T*A)\xrightarrow{X(u(A,T))}X(A)$$
est l'identité, où $u(A,T) : A\to T*A$ est l'inclusion canonique ;
\item étant donnés $\varphi : A\to B$, $T$ et $\tau : T\to T*B$ dans $\gr$ de sorte que le morphisme $\theta : T*B\to T*B$ égal à l'identité sur $B$ et à $\tau$ sur $T$ soit un isomorphisme, si l'on note $\psi : T*A\to T*B$ le morphisme de composantes $A\xrightarrow{\varphi} B\xrightarrow{u(B,T)}T*B$ et $T\xrightarrow{\tau}T*B$, alors les composées
$$X(B)\xrightarrow{X(\varphi)}X(A)\xrightarrow{\xi(A,T)}X(T*A)$$
et
$$X(B)\xrightarrow{\xi(B,T)}X(T*B)\xrightarrow{X(\psi)}X(T*A)$$
coïncident.
\end{enumerate}

Alors ${\rm Tor}^\gr_*(X,\abl\otimes P_r)=0$ pour tout entier $r$.
\end{pr}

(Remarquer que la première propriété est un cas particulier de la troisième, mais il nous est plus commode pour la suite de les différencier.)

\begin{proof}
 Pour tout entier $n>0$, posons
$$h_n=\xi(\mathbb{Z}^{*n+r},\mathbb{Z}) : X(n+r)\to X(n+r+1).$$

Grâce au corollaire~\ref{cor-bar}, il suffit de prouver la relation d'homotopie $\delta_n h_n+h_{n-1}\delta_{n-1}= Id_{X(n+r)}$ pour tout entier $n>0$ (à lire $\delta_1 h_1=Id$ pour $n=1$). Celle-ci découle des identités suivantes :
\begin{enumerate}
 \item $X(a^{n,r})h_n=Id_{X(n+r)}$ ;
\item $X(b_1^{n,r})h_n=h_{n-1}X(a^{n-1,r})$ ;
\item $X(b_{i+1}^{n,r})h_n=h_{n-1}X(b_i^{n-1,r})$ pour $1\leq i\leq n-1$ ;
\item $X(c^{n,r})h_n=h_{n-1} X(c^{n-1,r})$
\end{enumerate}
qu'on montre maintenant.

La première égalité provient de la deuxième hypothèse sur les morphismes $\xi$, du fait que $a^{n,r}=u(\mathbb{Z}^{*n+r},\mathbb{Z})$.

La troisième (resp. quatrième) égalité provient de la première hypothèse sur les morphismes $\xi$, puisque $b_{i+1}^{n,r}=\mathbb{Z}*b_i^{n,r}$ (resp. $c^{n,r}=\mathbb{Z}*c^{n-1,r}$).

Pour la deuxième égalité, on applique la troisième hypothèse sur les $\xi$ avec $\varphi=a^{n-1,r}=u(\mathbb{Z}^{*n+r-1},\mathbb{Z}) : \mathbb{Z}^{*n+r-1}\to\mathbb{Z}^{*n+r}$ et $T=\mathbb{Z}$, en prenant pour $\tau : \mathbb{Z}\to\mathbb{Z}*\mathbb{Z}^{*n+r}=\mathbb{Z}^{*n+r+1}$ le morphisme donné par l'élément $e_1 e_2$ du but. La condition d'inversibilité de $\theta$ est clairement vérifiée (c'est un générateur canonique du groupe des automorphismes de $\mathbb{Z}^{*n+r+1}$), et $\psi$ n'est autre que $b_1^{n,r}$. Cela termine la démonstration.
\end{proof}

\section{Comportement homologique du foncteur $i : \G\to\gr$ sur les foncteurs polynomiaux} \label{section5}

La proposition~\ref{pr-dvgl} et le corollaire~\ref{cor-dvg} montrent l'intérêt de savoir calculer des groupes d'homologie $H_*(\G;F)$, mais ceux-ci ne sont pas accessibles directement. La section précédente suggère de transiter par la catégorie plus usuelle $\gr$ pour mener à bien certains de ces calculs.

D'un point de vue formel, le foncteur $i : \G\to\gr$ induit, pour tout $F\in {\rm Ob}\,\fct(\gr,\kk\md)$, un morphisme naturel $H_*(\G;i^*F)\to H_*(\gr;F)$. Ce dernier groupe est réduit à $F(0)$ concentré en degré nul, puisque le foncteur constant $\kk$ sur $\gr^{op}$, égal à $P^{\gr^{op}}_0$ (puisque $0$ est objet final de $\gr$), est projectif et représente l'évaluation en $0$. On va voir que ce morphisme est en fait un isomorphisme lorsque $F$ est polynomial. Pour cela, on considère la suite spectrale de Grothendieck associée à la composée
$$\fct(\G^{op}, \kk\md) \xrightarrow{i_!} \fct(\gr^{op}, \kk\md) \xrightarrow{-\underset{\gr}{\otimes} F} \kk\md$$
où $i_!$ est l'extension de Kan à gauche du foncteur $i$, qui prend la forme suivante (cf. par exemple \cite{Dja-JKT}, appendice A.1, Remarque A.2) :
\begin{equation}\label{ssg}
E^2_{p,q}={\rm Tor}^\gr_p(L_q,F)\Rightarrow H_{p+q}(\G;i^*F)
\end{equation}
où $L_q : \gr^{op}\to\kk\md$ est le foncteur associant à l'objet $A$ de $\gr$ le $q$-ème groupe d'homologie de la catégorie $\G[A]$ définie comme suit :

\begin{defi}
Soit $A \in {\rm Ob}\, \gr$. On note $\G[A]$ la catégorie dont les objets sont les couples $(B,f)$ formés d'un objet $B$ de $\G$ et d'un élément $f$ de $\gr(A, i(B))$, les morphismes $(B,f)\to (B',f')$ étant les morphismes $u : B\to B'$ de $\G$ tels que $f'=i(u)\circ f\in\mathbf{gr}(A,i(B'))$.
\end{defi}

Nous montrons, dans la suite de cette section, que la deuxième page de cette suite spectrale est nulle lorsque $F$ est un foncteur polynomial réduit.

Soit $T \in {\rm Ob}\,\gr$, $i$ étant monoïdal, on a un foncteur : $T* - :\G\to\G$ qui induit un foncteur $\G[A]\to\G[T*A]$.

\begin{lm}\label{lmf}
Pour tout entier naturel $r$, on a : 
 $${\rm Tor}_*^{\gr}(L_\bullet, \abl \otimes P_r)=0.$$
\end{lm}

\begin{proof}
Par la proposition \ref{pr-ancoabst} (dont on conserve les notations), il suffit de vérifier que les applications linéaires 
$$\xi(A,T) : L_\bullet(A)\to L_\bullet(T*A)$$
induites en homologie par les foncteurs $\G[A]\xrightarrow{T* - }\G[T*A]$ satisfont les trois propriétés de l'énoncé sus-cité.

La première propriété provient de ce que les morphismes
$$L_\bullet(B)\xrightarrow{L_\bullet(\varphi)} L_\bullet(A)\xrightarrow{\xi(A,T)}L_\bullet(T*A)$$
et
$$L_\bullet(B)\xrightarrow{\xi(B,T)} L_\bullet(T*B)\xrightarrow{L_\bullet(T*\varphi)}L_\bullet(T*A)$$
sont induits par les foncteurs $\G[B]\to\G[T*A]$
$$\big(G\in {\rm Ob}\,\G,f\in\gr(B,iG)\big)\mapsto \big(T*G,T*(A\xrightarrow{\varphi}B\xrightarrow{f}iG)\big)$$
et 
$$\big(G\in {\rm Ob}\,\G,f\in\gr(B,iG)\big)\mapsto \big(T*G,T*A\xrightarrow{T*\varphi}T*B\xrightarrow{T*f}T*iG\big)$$
respectivement, qui sont égaux.

Pour la deuxième, on note que la composée
$$L_\bullet(A)\xrightarrow{\xi(A,T)}L_\bullet(T*A)\xrightarrow{L_\bullet(u(A,T))}L_\bullet(A)$$
est induite par l'endofoncteur
$$\big(G\in {\rm Ob}\,\G,f\in\gr(A,iG)\big)\mapsto \big(T*G,A\xrightarrow{u(A,T)}T*A\xrightarrow{T*f}T*iG)\big)$$
de $\G[A]$. Comme le diagramme
$$\xymatrix{A\ar[rr]^-{u(A,T)}\ar[d]_-f & &  T*A\ar[d]^-{T*f} \\
iG\ar[rr]^-{u(iG,T)} & & T*iG
}$$
commute, le morphisme naturel $G\to T*G$ de $\G$ induit une transformation naturelle de l'identité de $\G[A]$ vers ce foncteur. Par conséquent (cf. \cite{QK}), celui-ci induit l'identité en homologie.  

Venons-en à la dernière propriété. Les composées
$$L_\bullet(B)\xrightarrow{X(\varphi)}L_\bullet(A)\xrightarrow{\xi(A,T)}L_\bullet(T*A)$$
et
$$L_\bullet(B)\xrightarrow{\xi(B,T)}L_\bullet(T*B)\xrightarrow{L_\bullet(\psi)}L_\bullet(T*A)$$
qui nous intéressent sont induites par les foncteurs $\G[B]\to\G[T*A]$
$$\big(G\in {\rm Ob}\,\G,f\in\gr(B,iG)\big)\mapsto \big(T*G,T*A\xrightarrow{T*\varphi}T*B\xrightarrow{T*f}T*iG)\big)$$
et 
$$\big(G\in {\rm Ob}\,\G,f\in\gr(B,iG)\big)\mapsto \big(T*G,T*A\xrightarrow{\psi}T*B\xrightarrow{T*f}T*iG\big)$$
respectivement. Le morphisme de groupes $T*G\to T*G$ dont les composantes sont $T\xrightarrow{\tau}T*B\xrightarrow{T*f}T*G$ et $G\xrightarrow{u(G,T)}T*G$ est un isomorphisme (son inverse est le morphisme de composantes $T\xrightarrow{\tau'}T*B\xrightarrow{T*f}T*G$, où $\tau'$ est la composante $T\to T*B$ de l'inverse de l'automorphisme $\theta$, et $u(G,T)$), c'est donc aussi un automorphisme de $T*G$ {\em dans la catégorie $\G$}, automorphisme que nous noterons $\gamma_G$. La conclusion provient alors des deux observations suivantes :
\begin{enumerate}
 \item le morphisme $\gamma_G$ est naturel en $G\in {\rm Ob}\,\G$ (il suffit pour le voir d'écrire ce que sont les morphismes dans $\G[B]$) ; 
\item il fait commuter le diagramme 
$$\xymatrix{T*A\ar[r]^-{T*\varphi}\ar[rd]_-\psi & T*B\ar[r]^-{T*f}\ar[d]^-\theta & T*G\ar[d]^-{\gamma_G} \\
& T*B\ar[r]^-{T*f} & T*G
}$$
de $\gr$, de sorte qu'il définit une transformation naturelle entre nos deux foncteurs $\G[B]\to\G[T*A]$, qui induisent donc la même application en homologie.
\end{enumerate}

\end{proof}

Nous avons également besoin du résultat suivant, totalement indépendant des considérations précédentes (il n'utilise que les résultats de structure de la première section sur les foncteurs polynomiaux).

\begin{pr}\label{praux-fctpol}
 Soit $X\in {\rm Ob}\,\fct(\gr^{op},\kk\md)$ tel que 
$${\rm Tor}^\gr_*(X,\abl\otimes P_n)=0$$
pour tout $n\in\mathbb{N}$. Alors 
$${\rm Tor}^\gr_*(X,F\otimes G)=0$$
pour tout $F\in {\rm Ob}\,\fct(\gr,\kk\md)$ polynomial réduit (i.e. tel que $F(0)=0$) et tout $G\in {\rm Ob}\,\fct(\gr,\kk\md)$.
\end{pr}

\begin{proof}
 La formule des coefficients universels (cf. par exemple \cite{Wei}, th.~3.6.1) montre qu'il suffit de traiter le cas où $\kk$ est un corps, ce qu'on suppose désormais. Ainsi, le produit tensoriel est exact en chaque variable.

On commence par montrer que
\begin{equation}\label{eq-dp1}
 {\rm Tor}^\gr_*(X,\abl\otimes G)=0
\end{equation}
pour tout $G\in {\rm Ob}\,\fct(\gr,\kk\md)$. Cela découle de l'hypothèse, de ce que $G$ possède une résolution par des sommes directes de projectifs $P_n$ et de la suite spectrale d'hyperhomologie associée (on renvoie à \cite{Wei}, §\,5.7 pour ce qui concerne l'hyperhomologie).

On établit maintenant par récurrence sur $d\in\mathbb{N}$ l'assertion suivante : {\em pour tout $F\in {\rm Ob}\,\fct(\gr,\kk\md)$ polynomial réduit de degré inférieur ou égal à $d$ et tout $G\in {\rm Ob}\,\fct(\gr,\kk\md)$, on a ${\rm Tor}^\gr_*(X,F\otimes G)=0$}.

Pour $d=0$ l'assertion est vide, on suppose donc $d>0$ et le résultat démontré jusqu'en degré $d-1$. Soit $F$ polynomial réduit de degré au plus $d$, notons $M$ le $\kk[\mathfrak{S}_d]$-module ${\rm cr}_d(F)$
 : d'après le corollaire~\ref{cor-recdeg}, le noyau $N$ et le conoyau $C$ du morphisme naturel $\alpha_d(M)=\abl^{\otimes d}\underset{\mathfrak{S}_d}{\otimes} M\to F$ sont de degré au plus $d-1$. Ainsi, ${\rm Tor}^\gr_*(X,N\otimes G)$ et ${\rm Tor}^\gr_*(X,C\otimes G)$ sont nuls pour tout $G$, par l'hypothèse de récurrence. Il s'ensuit que la nullité de ${\rm Tor}^\gr_*(X,F\otimes G)$ est équivalente à celle de ${\rm Tor}^\gr_*(X,\alpha_d(M)\otimes G)$.

Celle-ci est acquise lorsque $M$ est un $\kk[\mathfrak{S}_d]$-module {\em libre}, car alors $\alpha_d(M)$ est somme directe de copies de $\abl^{\otimes d}$, et 
$${\rm Tor}_*^\gr(X,\abl^{\otimes d}\otimes G)={\rm Tor}_*^\gr(X,\abl\otimes(\abl^{\otimes d-1}\otimes G))=0$$
par~(\ref{eq-dp1}).

Considérons, dans le cas général, une résolution libre $L_\bullet\to M$ du $\kk[\mathfrak{S}_d]$-module $M$ et notons $H_\bullet$ l'homologie du complexe $\alpha_d(L_\bullet)\to\alpha_d(M)$ de $\pol_d(\gr)$. Comme le foncteur canonique $\pi_d : \pol_d(\gr)\to\pol_d(\gr)/\pol_{d-1}(\gr)$ est exact, l'homologie de l'image par $\pi_d$ de ce complexe s'identifie à $\pi_d(H_\bullet)$. Mais comme le foncteur
$$\kk[\mathfrak{S}_d]\md\xrightarrow{\alpha_d}\pol_d(\gr)\xrightarrow{\pi_d}\pol_d(\gr)/\pol_{d-1}(\gr)$$
est une équivalence de catégories d'après le théorème~\ref{th-pirarec}, donc en particulier un foncteur exact, l'exactitude de $L_\bullet\to M$ implique que $\pi_d(H_\bullet)=0$. Autrement dit, le foncteur gradué $H_\bullet$ est de degré au plus $d-1$. Comme il est également réduit, l'hypothèse de récurrence montre que 
$${\rm Tor}_*^\gr(X,H_\bullet\otimes G)=0$$
pour tout $G$.

On en déduit que l'hyperhomologie $X\underset{\gr}{\overset{\mathbf{L}}{\otimes}} (\alpha_d(L_\bullet)\otimes G)$ est isomorphe à ${\rm Tor}^\gr_*(X,\alpha_d(M)\otimes G)$, par l'une des suites spectrales d'hyperhomologie associées. L'autre suite spectrale d'hyperhomologie associée a pour première page ${\rm Tor}^\gr_*(X,\alpha_d(L_\bullet)\otimes G)$, qui est identiquement nul par ce qu'on a vu plus haut, parce que les $L_i$ sont libres.

Par conséquent, ${\rm Tor}^\gr_*(X,\alpha_d(M)\otimes G)$ et donc ${\rm Tor}^\gr_*(X,F\otimes G)$ sont nuls, ce qui achève la démonstration.

\end{proof}

Nous pouvons désormais démontrer le résultat principal de ce travail.

\begin{thm}\label{thf}
 Soient $F$ et $G$ des foncteurs de $\fct(\gr,\kk\md)$ ; on suppose $F$ polynomial et réduit. Alors
$$\underset{n\in\mathbb{N}}{\col}H_*({\rm Aut}\,(\mathbb{Z}^{*n});(F\otimes G)(\mathbb{Z}^{*n}))=0.$$
\end{thm}

\begin{proof}
 Le lemme~\ref{lmf} et la proposition~\ref{praux-fctpol} montrent que
$${\rm Tor}^\gr_*(L_\bullet,F\otimes G)=0.$$
La suite spectrale~(\ref{ssg}) permet d'en déduire la nullité de $H_*(\G;i^* (F\otimes G))$, qui entraîne celle de $H_*(\G\times G_\infty;i^* (F\otimes G))$ par la formule de Künneth. On conclut alors par la proposition~\ref{pr-dvgl}.
\end{proof}

Rappelons que, dans le cas crucial où les coefficients sont tordus par le foncteur d'abélianisation $\abl$, ce résultat (c'est-à-dire $\underset{n\in\mathbb{N}}{\col} H_*({\rm Aut}\,(\mathbb{Z}^{*n});\abl(\mathbb{Z}^{*n}))=0$) a été antérieurement obtenu par Hatcher-Wahl dans \cite{HW} (et son erratum \cite{HW-erra}) par des méthodes complètement différentes qui permettent également d'obtenir une borne de stabilité homologique explicite. Dans le cas des degrés homologiques $1$ et $2$, ce résultat a été aussi obtenu par Satoh dans \cite{Sat1} et \cite{Sat2} (avec une borne de stabilité meilleure que celle de \cite{HW}), suivant une approche différente de celles d'Hatcher-Wahl et du présent article.

Rappelons également que le théorème~\ref{thf} peut s'obtenir pour toutes les puissances tensorielles du foncteur d'abélianisation par des méthodes géométriques analogues à celles du travail d'Hatcher-Wahl susmentionné : l'article \cite{HV04} d'Hatcher et Vogtmann (et son erratum \cite{HVW} avec Wahl) entraîne facilement ce résultat (avec de plus un résultat de stabilité homologique), comme l'a noté Randal-Williams dans \cite{RW} (proposition~1.3). Rationnellement (i.e. pour $\kk=\mathbb{Q}$), on peut en déduire facilement le résultat pour tous les foncteurs polynomiaux réduits (i.e. le théorème~\ref{thf} pour $G$ constant --- le corollaire~1.4 de \cite{RW} en donne un cas particulier), en utilisant la structure des foncteurs polynomiaux et le caractère semi-simple des représentations rationnelles des groupes symétriques, mais (même pour $G$ constant) le théorème~\ref{thf} ne semble pas pouvoir s'obtenir facilement par ces méthodes lorsqu'on travaille sur $\kk=\mathbb{Z}$.

\section{Comportement homologique en degré $1$ du foncteur $\iota : \G^{op}\to\gr$ sur les foncteurs polynomiaux}

La comparaison de l'homologie de $\G$ et de $\gr^{op}$ via le foncteur $\iota$, à coefficients dans un foncteur (même polynomial) $F$ sur $\gr^{op}$, ne fonctionne pas de la même façon qu'avec le foncteur $i$ : $H_*(\gr^{op};F)$ est toujours réduit à $F(0)$ concentré en degré nul (puisque $0$ est objet initial de $\gr$), mais l'annulation ne subsiste plus pour l'homologie de $H_*(\G;\iota^*F)$. Pour l'instant, les auteurs ne savent pas comment calculer ces groupes d'homologie des foncteurs. Toutefois, on peut résoudre le problème d'une manière différente, plus directe, pour le degré homologique $1$ (en degré homologique nul, le morphisme qu'induit $\iota$ est toujours un isomorphisme), au moins pour un foncteur polynomial se factorisant par le foncteur d'abélianisation $\gr\to\mathbf{ab}$.

\begin{rem} \label{Rem-Betley}
 Cette situation contraste avec celle qu'on rencontre pour les groupes linéaires. Soit en effet un foncteur $F : \mathbf{ab}^{op}\to\mathbf{Ab}$, notons $G : \mathbf{ab}\to\mathbf{Ab}$ le foncteur obtenu en précomposant $F$ par la dualité ${\rm Hom}_\mathbb{Z}(-, \mathbb{Z}) : \mathbf{ab}\to\mathbf{ab}^{op}$. Pour tout entier $n$, l'automorphisme involutif $g\mapsto\,^t\negmedspace g^{-1}$ de $GL_n(\mathbb{Z})$ induit un isomorphisme $H_*(GL_n(\mathbb{Z});F(\mathbb{Z}^n))\xrightarrow{\simeq} H_*(GL_n(\mathbb{Z});G(\mathbb{Z}^n))$ (et ces isomorphismes, lorsque $n$ varie, sont compatibles aux applications de stabilisation), de sorte que l'homologie stable des groupes linéaires sur $\mathbb{Z}$ à coefficients dans $F$ s'identifie à l'homologie stable de ces groupes à coefficients dans $G$. On ne peut pas étendre ce raisonnement aux groupes d'automorphismes des groupes libres parce que l'automorphisme $g\mapsto\,^t\negmedspace g^{-1}$ de $GL_n(\mathbb{Z})$ ne se relève pas en un automorphisme de ${\rm Aut}\,(\mathbb{Z}^{*n})$ pour la projection canonique ${\rm Aut}\,(\mathbb{Z}^{*n})\twoheadrightarrow GL_n(\mathbb{Z})$.
\end{rem}

Le résultat qui suit est essentiellement présent dans le travail \cite{K-Magnus} de Ka\-wa\-zu\-mi (voir la fin de sa section 6), avec l'hypothèse que les coefficients sont des $\mathbb{Q}$-espaces vectoriels (et sans le langage des catégories de foncteurs). Notre démonstration suit celle de Kawazumi, hormis pour les arguments d'annulation pour lesquels nous invoquons des résultats sur l'homologie des foncteurs.

\begin{pr}\label{pr-contra1}
 Soit $F : \mathbf{ab}^{op}\to\mathbf{Ab}$ un foncteur polynomial réduit. Il existe un isomorphisme naturel
$$\underset{n\in\mathbb{N}}{\col} H_1({\rm Aut}\,(\mathbb{Z}^{*n});F(\mathbb{Z}^n))\simeq F\underset{\mathbf{ab}}{\otimes}\mathrm{Id}$$
où ${\rm Aut}\,(\mathbb{Z}^{*n})$ opère sur $F(\mathbb{Z}^n)$ via la projection sur $GL_n(\mathbb{Z})$ et $\mathrm{Id} : \mathbf{ab}\to\mathbf{Ab}$ désigne le foncteur d'inclusion.
\end{pr}

\begin{proof}
 Notons $IA(G)$, pour $G$ un groupe libre, le noyau de l'épimorphisme ${\rm Aut}_\gr(G)\twoheadrightarrow {\rm Aut}_\mathbf{ab}(G_{ab})$ induit par l'abélianisation. Le théorème~6.1 de \cite{K-Magnus} procure un isomorphisme ${\rm Aut}_\mathbf{ab}(G_{ab})$-équivariant
$$H_1(IA(G))\simeq\mathbf{Ab}(G_{ab},\Lambda^2(G_{ab}))$$
où $\Lambda^2$ désigne la deuxième puissance extérieure.

La suite exacte à cinq termes (cf. par exemple \cite{Wei}, §\,6.8.3) associée à l'extension de groupes
$$1\to IA(G)\to {\rm Aut}_\gr(G)\to {\rm Aut}_\mathbf{ab}(G_{ab})\to 1$$
fournit donc une suite exacte
\begin{eqnarray*}
\underset{n\in\mathbb{N}}{\col}H_2(GL_n(\mathbb{Z});F(\mathbb{Z}^n))\to \underset{n\in\mathbb{N}}{\col}H_0\big(GL_n(\mathbb{Z});F(\mathbb{Z}^n)\otimes\mathbf{Ab}(\mathbb{Z}^n,\Lambda^2(\mathbb{Z}^n))\big)\\
\to \underset{n\in\mathbb{N}}{\col} H_1({\rm Aut}\,(\mathbb{Z}^{*n});F(\mathbb{Z}^n))\to \underset{n\in\mathbb{N}}{\col} H_1(GL_n(\mathbb{Z});F(\mathbb{Z}^n)).
\end{eqnarray*}

En effet, comme ${\rm Aut}\,(\mathbb{Z}^{*n})$ agit sur $F(\mathbb{Z}^n)$ via la projection sur $GL_n(\mathbb{Z})$, $IA( \mathbb{Z}^{*n})$ agit trivialement sur $F(\mathbb{Z}^n)$, d'où $H_0\big(IA( \mathbb{Z}^{*n}) ; F(\mathbb{Z}^n)\big)=F(\mathbb{Z}^n)$ et
$$H_0\big(GL_n(\mathbb{Z}); H_1(IA( \mathbb{Z}^{*n}); F(\mathbb{Z}^n))\big)=H_0\big(GL_n(\mathbb{Z}); H_1(IA( \mathbb{Z}^{*n})) \otimes  F(\mathbb{Z}^n)\big).$$
Par un théorème de Betley (cf. \cite{Bet2}, théorème~4.2) et la remarque \ref{Rem-Betley} on a :
 $\underset{n\in\mathbb{N}}{\col}H_*(GL_n(\mathbb{Z});F(\mathbb{Z}^n))=0$ quand $F$ est un foncteur contravariant polynomial réduit.

Par ailleurs, d'après un cas particulier facile du théorème de Scorichenko\,\footnote{Cf. \cite{Sco}, ou \cite{Dja-JKT} pour une reprise publiée de ce résultat. Mais le cas de degré nul dont on a seul besoin ici est beaucoup plus simple : sa démonstration ne nécessite aucun résultat subtil d'annulation en homologie des foncteurs.} étendant le résultat de Betley, on a, pour un bifoncteur $B: \mathbf{ab}^{op} \times \mathbf{ab} \to \mathbf{Ab}$, un isomorphisme naturel :
$$\underset{n\in\mathbb{N}}{\col}H_0\big(GL_n(\mathbb{Z});B(\mathbb{Z}^n, \mathbb{Z}^n )\big)\simeq \mathbb{Z}[\mathbf{ab}^{op}] \underset{\mathbf{ab}^{op} \times \mathbf{ab}}{\otimes} B.$$
En appliquant ce résultat au bifoncteur $B$ défini par 
$$B(U,V)=F(U) \otimes \mathbf{Ab}(U, \Lambda^2(V))$$
qui est naturellement isomorphe à $(F\otimes \mathrm{Id}^{\vee}) \boxtimes \Lambda^2$ (où $\boxtimes$ désigne le produit tensoriel extérieur et l'exposant $\vee$ la dualité des groupes abéliens libres)
et en utilisant l'isomorphisme classique de groupes abéliens 
$$ \mathbb{Z}[\mathbf{ab}^{op}] \underset{\mathbf{ab}^{op} \times \mathbf{ab}}{\otimes}  \left( (F\otimes \mathrm{Id}^{\vee}) \boxtimes \Lambda^2 \right)  \simeq (F\otimes \mathrm{Id}^{\vee}) \underset{ \mathbf{ab}}{\otimes} \Lambda^2$$
on obtient un isomorphisme naturel $$\underset{n\in\mathbb{N}}{\col}H_0\big(GL_n(\mathbb{Z});F(\mathbb{Z}^n)\otimes\mathbf{Ab}(\mathbb{Z}^n,\Lambda^2(\mathbb{Z}^n))\big)\simeq(F\otimes\mathrm{Id}^\vee)\underset{\mathbf{ab}}{\otimes}\Lambda^2\;;$$
il s'agit donc de montrer que ce groupe est isomorphe à $F\underset{\mathbf{ab}}{\otimes}\mathrm{Id}$.

Le foncteur gradué puissance extérieure $\Lambda^\bullet=(\Lambda^i)_{i \in \mathbb{N}}$ étant exponentiel (i.e. il  transforme sommes directes en produits tensoriels, au sens gradué), on a un isomorphisme naturel de groupes abéliens gradués :
\begin{equation} \label{Tor}
\mathrm{Tor}_*(F \otimes \mathrm{Id}^\vee, \Lambda^2) =\underset{i+j=2}{\bigoplus} \mathrm{Tor}_*(F, \Lambda^i) \otimes \mathrm{Tor}_*(\mathrm{Id}^\vee, \Lambda^j)
\end{equation}
(cf. par exemple \cite{Franjou96}, proposition~1.4.2\,\footnote{Cet article parle de groupes d'extensions ; la propriété analogue en termes de groupes de torsion est encore plus facile, a fortiori le seul degré nul dont on a ici besoin.}). Comme $\mathrm{Tor}_0(F, \Lambda^0)=F(0)=0$ et $\mathrm{Tor}_0( \mathrm{Id}^\vee, \Lambda^0)=0$, on déduit de (\ref{Tor}) l'isomorphisme :
$$(F\otimes\mathrm{Id}^\vee)\underset{\mathbf{ab}}{\otimes}\Lambda^2=\mathrm{Tor}_0(F\otimes\mathrm{Id}^\vee, \Lambda^2)=\mathrm{Tor}_0(F, \mathrm{Id}) \otimes \mathrm{Tor}_0(\mathrm{Id}^\vee, \mathrm{Id})=(F\underset{\mathbf{ab}}{\otimes}\mathrm{Id}) \otimes ( \mathrm{Id}^\vee \underset{\mathbf{ab}}{\otimes}\mathrm{Id}).$$
Or le groupe $\mathrm{Id}^\vee \underset{\mathbf{ab}}{\otimes}\mathrm{Id}$ est isomorphe à $\mathbb{Z}$.  En effet, comme le foncteur $\mathrm{Id}$ est un quotient de $P^{\mathbf{ab}}_{\mathbb{Z}}$, en appliquant le foncteur $\mathrm{Id}^\vee \underset{\mathbf{ab}}{\otimes}-$, exact à droite, on obtient une suite exacte :

\begin{equation*} \label{se}
 \mathrm{Id}^\vee \underset{\mathbf{ab}}{\otimes}P^{\mathbf{ab}}_{\mathbb{Z}} \to \mathrm{Id}^\vee \underset{\mathbf{ab}}{\otimes}\mathrm{Id} \to 0.
\end{equation*}
En utilisant l'isomorphisme  $\mathrm{Id}^\vee \underset{\mathbf{ab}}{\otimes}P^{\mathbf{ab}}_{\mathbb{Z}} \simeq \mathbb{Z}$, on montre que $\mathrm{Id}^\vee \underset{\mathbf{ab}}{\otimes}\mathrm{Id} \simeq \mathbb{Z}$.
 \end{proof}

\begin{rem}
 Alors qu'elle ne fournit qu'un résultat très partiel, la démonstration de la proposition précédente utilise beaucoup plus de résultats de théorie des groupes que celle du théorème~\ref{thf} : l'identification de l'abélianisation des sous-groupes $IA$ des automorphismes des groupes libres nécessite par exemple (cf. \cite{K-Magnus}) la connaissance d'un <<~bon~>> ensemble de générateurs de ces groupes. Celle-ci est ancienne (elle remonte aux travaux de Magnus des années 1930), mais absolument pas immédiate ; notre démonstration du théorème~\ref{thf} ne nécessite même pas la connaissance d'un ensemble de générateurs des automorphismes des groupes libres. Cela illustre la puissance des méthodes d'homologie des foncteurs.

Satoh (cf. \cite{Sat1} et \cite{Sat2}) a utilisé la structure fine des groupes d'automorphismes des groupes libres à des fins homologiques. Il montre notamment les résultats suivants : $H_i({\rm Aut}(\mathbb{Z}^{*n});\mathbb{Z}^n)=0$ pour $i=1$ et $n\geq 4$, ou $i=2$ et $n\geq 6$ (ce qui redonne un cas particulier du théorème~\ref{thfi} pour le foncteur d'abélianisation), $H_1({\rm Aut}(\mathbb{Z}^{*n});(\mathbb{Z}^n)^\vee)\simeq\mathbb{Z}$ pour $n\geq 4$ (ce qui redonne l'assertion de la proposition~\ref{pr-contra1} pour le dual du foncteur d'abélianisation) et $H_2({\rm Aut}(\mathbb{Z}^{*n});(\mathbb{Z}^n)^\vee)=0$ pour $n\geq 6$. L'avantage de ces méthodes est d'obtenir également des bornes de stabilité et certains calculs instables (i.e. pour $n$ petit), mais il semble très difficile d'obtenir de la sorte l'annulation homologique stable en {\em tout} degré homologique, même en se cantonnant aux coefficients tordus par le foncteur d'abélianisation.
\end{rem}

Outre étendre la proposition~\ref{pr-contra1} au cas du degré homologique supérieur à $1$ en trouvant un cadre fonctoriel approprié, une question très naturelle est de comprendre la situation plus générale des {\em bi}foncteurs sur $\gr$, i.e. d'étudier le comportement homologique du foncteur $\G\xrightarrow{(\iota^{op},i)}\gr^{op}\times\gr$ à coefficients polynomiaux. En effet, d'une part, la situation connue pour les groupes linéaires (cf. les travaux de Scorichenko susmentionnés) conduit à cette question ; d'autre part, Kawazumi a introduit dans \cite{K-Magnus} (section 4) des classes de cohomologie tordues par des bifoncteurs --- provenant en fait de bifoncteurs sur $\mathbf{ab}$ via l'abélianisation.

\bibliographystyle{plain}
\bibliography{bibli-gl.bib}

\begin{thebibliography}{10}

\bibitem{Bet2}
Stanislaw Betley.
\newblock Homology of {${\rm Gl}(R)$} with coefficients in a functor of finite
  degree.
\newblock {\em J. Algebra}, 150(1):73--86, 1992.

\bibitem{Bet}
Stanislaw Betley.
\newblock Stable {$K$}-theory of finite fields.
\newblock {\em $K$-Theory}, 17(2):103--111, 1999.

\bibitem{Bet-sym}
Stanislaw Betley.
\newblock Twisted homology of symmetric groups.
\newblock {\em Proc. Amer. Math. Soc.}, 130(12):3439--3445 (electronic), 2002.

\bibitem{BorceuxI}
Francis Borceux.
\newblock {\em Handbook of categorical algebra. 1}, volume~50 of {\em
  Encyclopedia of Mathematics and its Applications}.
\newblock Cambridge University Press, Cambridge, 1994.
\newblock Basic category theory.

\bibitem{Dja-JKT}
Aur\'elien Djament.
\newblock Sur l'homologie des groupes unitaires à coefficients polynomiaux.
\newblock {\em J. K-Theory}, 10(1):87--139, 2012.

\bibitem{DV}
Aur{\'e}lien Djament and Christine Vespa.
\newblock Sur l'homologie des groupes orthogonaux et symplectiques \`a
  coefficients tordus.
\newblock {\em Ann. Sci. \'Ec. Norm. Sup\'er. (4)}, 43(3):395--459, 2010.

\bibitem{EML}
Samuel Eilenberg and Saunders Mac~Lane.
\newblock On the groups {$H(\Pi,n)$}. {II}. {M}ethods of computation.
\newblock {\em Ann. of Math. (2)}, 60:49--139, 1954.

\bibitem{Franjou96}
Vincent Franjou.
\newblock Extensions entre puissances ext\'erieures et entre puissances
  sym\'etriques.
\newblock {\em J. Algebra}, 179(2):501--522, 1996.

\bibitem{FFSS}
Vincent Franjou, Eric~M. Friedlander, Alexander Scorichenko, and Andrei Suslin.
\newblock General linear and functor cohomology over finite fields.
\newblock {\em Ann. of Math. (2)}, 150(2):663--728, 1999.

\bibitem{FLS}
Vincent Franjou, Jean Lannes, and Lionel Schwartz.
\newblock Autour de la cohomologie de {M}ac {L}ane des corps finis.
\newblock {\em Invent. Math.}, 115(3):513--538, 1994.

\bibitem{FPZ}
Vincent Franjou and Teimuraz Pirashvili.
\newblock On the {M}ac {L}ane cohomology for the ring of integers.
\newblock {\em Topology}, 37(1):109--114, 1998.

\bibitem{FP}
Vincent Franjou and Teimuraz Pirashvili.
\newblock Comparison of abelian categories recollements.
\newblock {\em Doc. Math.}, 9:41--56 (electronic), 2004.

\bibitem{Gal}
S{\o}ren Galatius.
\newblock Stable homology of automorphism groups of free groups.
\newblock {\em Ann. of Math. (2)}, 173(2):705--768, 2011.

\bibitem{HV}
Manfred Hartl and Christine Vespa.
\newblock Quadratic functors on pointed categories.
\newblock {\em Adv. Math.}, 226(5):3927--4010, 2011.

\bibitem{HVW}
Allan Hatcher, Karen Vogtmann, and Natalie Wahl.
\newblock Erratum to: ``{H}omology stability for outer automorphism groups of
  free groups [{A}lgebr. {G}eom. {T}opol. {\bf 4} (2004), 1253--1272
  (electronic)] by {H}atcher and {V}ogtmann.
\newblock {\em Algebr. Geom. Topol.}, 6:573--579 (electronic), 2006.

\bibitem{HVo}
Allen Hatcher and Karen Vogtmann.
\newblock Cerf theory for graphs.
\newblock {\em J. London Math. Soc. (2)}, 58(3):633--655, 1998.

\bibitem{HV04}
Allen Hatcher and Karen Vogtmann.
\newblock Homology stability for outer automorphism groups of free groups.
\newblock {\em Algebr. Geom. Topol.}, 4:1253--1272, 2004.

\bibitem{HW}
Allen Hatcher and Nathalie Wahl.
\newblock Stabilization for the automorphisms of free groups with boundaries.
\newblock {\em Geom. Topol.}, 9:1295--1336 (electronic), 2005.

\bibitem{HW-erra}
Allen Hatcher and Nathalie Wahl.
\newblock Erratum to: ``{S}tabilization for the automorphisms of free groups
  with boundaries'' [{G}eom. {T}opol. {\bf 9} (2005), 1295--1336; 2174267].
\newblock {\em Geom. Topol.}, 12(2):639--641, 2008.

\bibitem{HW10}
Allen Hatcher and Nathalie Wahl.
\newblock Stabilization for mapping class groups of 3-manifolds.
\newblock {\em Duke Math. J.}, 155(2):205--269, 2010.

\bibitem{PJ}
Mamuka Jibladze and Teimuraz Pirashvili.
\newblock Cohomology of algebraic theories.
\newblock {\em J. Algebra}, 137(2):253--296, 1991.

\bibitem{K-Magnus}
Nariya Kawazumi.
\newblock Cohomological aspects of magnus expansions.
\newblock arXiv : math.GT/0505497, 2006.

\bibitem{K2}
Nicholas~J. Kuhn.
\newblock Generic representations of the finite general linear groups and the
  {S}teenrod algebra. {II}.
\newblock {\em $K$-Theory}, 8(4):395--428, 1994.

\bibitem{Nak}
Minoru Nakaoka.
\newblock Decomposition theorem for homology groups of symmetric groups.
\newblock {\em Ann. of Math. (2)}, 71:16--42, 1960.

\bibitem{Ni}
Jakob Nielsen.
\newblock Die {I}somorphismengruppe der freien {G}ruppen.
\newblock {\em Math. Ann.}, 91(3-4):169--209, 1924.

\bibitem{Pet}
Alexandra Pettet.
\newblock The {J}ohnson homomorphism and the second cohomology of {${\rm
  IA}_n$}.
\newblock {\em Algebr. Geom. Topol.}, 5:725--740, 2005.

\bibitem{Pira-rec}
T.~I. Pirashvili.
\newblock Polynomial functors.
\newblock {\em Trudy Tbiliss. Mat. Inst. Razmadze Akad. Nauk Gruzin. SSR},
  91:55--66, 1988.

\bibitem{QK}
Daniel Quillen.
\newblock Higher algebraic {$K$}-theory. {I}.
\newblock In {\em Algebraic {$K$}-theory, {I}: {H}igher {$K$}-theories ({P}roc.
  {C}onf., {B}attelle {M}emorial {I}nst., {S}eattle, {W}ash., 1972)}, pages
  85--147. Lecture Notes in Math., Vol. 341. Springer, Berlin, 1973.

\bibitem{RW}
Oscar Randal-Williams.
\newblock The stable cohomology of automorphisms of free groups with
  coefficients in the homology representation.
\newblock arXiv : math.AT/1012.1433, 2010.

\bibitem{Sat1}
Takao Satoh.
\newblock Twisted first homology groups of the automorphism group of a free
  group.
\newblock {\em J. Pure Appl. Algebra}, 204(2):334--348, 2006.

\bibitem{Sat2}
Takao Satoh.
\newblock Twisted second homology groups of the automorphism group of a free
  group.
\newblock {\em J. Pure Appl. Algebra}, 211(2):547--565, 2007.

\bibitem{Sco}
Alexander Scorichenko.
\newblock {\em Stable {K}-theory and functor homology over a ring}.
\newblock PhD thesis, Evanston, 2000.

\bibitem{Wei}
Charles~A. Weibel.
\newblock {\em An introduction to homological algebra}, volume~38 of {\em
  Cambridge Studies in Advanced Mathematics}.
\newblock Cambridge University Press, Cambridge, 1994.

\end{thebibliography}
\end{document}